\documentclass[12pt,leqno]{amsart}
\usepackage{verbatim,amssymb,amsfonts,latexsym,amsmath,amsthm,tikz,enumitem}

\newtheorem{theorem}{Theorem}[section]
\newtheorem{lemma}[theorem]{Lemma}

\newtheorem{question}[theorem]{Question}
\theoremstyle{definition}

\theoremstyle{remark}

\newtheorem*{sub-claim}{sub-claim}

\newcommand{\Z}{\mathbb{Z}}
\newcommand{\R}{\mathbb{R}}
\newcommand{\C}{\mathcal{C}}

\newcommand{\explicitSet}[1]{\left\lbrace #1 \right\rbrace}

\newcommand{\set}[2]{\explicitSet{#1 \colon #2}}

\renewcommand{\>}{\rangle}

\newcommand{\card}[1]{\left\lvert #1 \right\rvert}

\newcommand{\p}{\mathbb{P}}

\newcommand{\w}{\omega}

\newcommand{\sub}{\subseteq}

\begin{document}

\title[Compact groups]{Every compact group can have a non-measurable subgroup}
\author{W. R. Brian}
\address {
William R. Brian\\
Department of Mathematics\\
Tulane University\\
6823 St. Charles Ave.\\
New Orleans, LA 70118}
\email{wbrian.math@gmail.com}
\author{M. W. Mislove}
\address {
Michael W. Mislove\\
Department of Mathematics\\
Tulane University\\
6823 St. Charles Ave.\\
New Orleans, LA 70118}
\email{mislove@tulane.edu}

\maketitle

\begin{abstract}
We show that it is consistent with ZFC that every compact group has a non-Haar-measurable subgroup. In addition, we demonstrate a natural construction, and we conjecture that this construction always produces a non-measurable subgroup of a given compact group. We prove that this is so in the Abelian case.
\end{abstract}

\section{The big question}

Every compact group admits a unique translation-invariant probability measure, namely its (normalized) Haar measure. In this note we discuss the question of whether every compact group has a non-Haar-measurable subgroup, henceforth referred to as \emph{the big question}. Our main result is that it is consistent with ZFC that the answer to the big question is yes.

The big question goes back at least as far as 1985 (see \cite{S&S}). A positive answer to the big question was given for compact Abelian groups by Comfort, Raczkowski, and Trigos-Arrieta in \cite{CRT}, and a more thorough analysis of non-measurable subgroups of the real line is given in \cite{Krz}. Partial progress on the non-commutative case was made in \cite{Gel}, and a good deal of further progress was made in \cite{HHM}. 

With the exception of Karazishvili's paper \cite{Krz}, these results have been accomplished by finding subgroups of countable index. If a subgroup of a compact group $G$ has index $\aleph_0$, then it is non-measurable (by the translation invariance and countable additivity of Haar measure). If a subgroup has finite index but is not closed, then it is non-measurable (see \cite{HHM}, Proposition 1.1(b)).

Hunting for countable index subgroups has been a powerful tool for answering the big question: in \cite{HHM}, this technique is used to solve every case except for a certain class of metric profinite groups. However, this last remaining case cannot be solved simply by finding countable index subgroups, because some groups in this class do not have non-closed subgroups of countable index. Therefore a new construction for finding non-measurable subgroups wil be needed before the big question can be put to rest.

In Section~\ref{sec:consistency} we will prove our main theorem. In Section~\ref{sec:constructions} we will give a natural construction for obtaining subgroups of a given compact group. For Abelian groups, we prove that these subgroups are non-measurable, and we conjecture that they are also non-measurable for arbitrary metric profinite groups as well. Proving this conjecture correct would result in a positive general solution to the big question.

\section{The main theorem}\label{sec:consistency}

The proof of our main theorem requires four ``lemmas'', each of which is an important theorem in its own right.

\begin{lemma}[Hern\'andez, Hoffman, and Morris]\label{HHMtheorem}
If $G$ is an infinite compact group other than a metric profinite group, then $G$ has a nonmeasurable subgroup.
\end{lemma}
\begin{proof}
This is the main result of \cite{HHM}.
\end{proof}

\begin{lemma}\label{CantorSpace}
Every infinite, metric, profinite group is, as a topological space, homeomorphic to $2^\w$.
\end{lemma}
\begin{proof}
See Theorem 10.40 in \cite{H&M}.
\end{proof}

For the statement of the next lemma, two measures $\mu$ and $\nu$ on a set $X$ are \textbf{isomorphic} if there is a bijection $\varphi: X \to X$ such that, for every $Y \sub X$, $Y$ is $\mu$-measurable if and only if $\varphi(Y)$ is $\nu$-measurable, and if this is the case then $\mu(Y) = \nu(\varphi(Y))$. By the Lebesgue measure on $2^\w$ we mean the unique measure generated by giving the set $\set{f \in 2^\w}{f(n) = 0}$ measure $\frac{1}{2}$ for each $n$.

\begin{lemma}\label{CantorSpace2}
The (normalized) Haar measure for any group structure on $2^\w$ is isomorphic to the Lebesgue measure.
\end{lemma}
\begin{proof}
This follows from Theorem 17.41 in \cite{Kec} together with the fact that the Haar measure for any group structure on $2^\w$ is continuous. A stronger version of this result, with a more constructive proof, is given in \cite{B&M}.
\end{proof}

\begin{lemma}\label{randomreals}
There is a model of \emph{ZFC} in which the following holds: there is a subset $X$ of $2^\w$ such that $X$ is not Lebesgue measurable and $\card{X} < 2^{\aleph_0}$.
\end{lemma}
\begin{proof}
This result is well-known. The idea is essentially this: if $M$ is a model of ZFC and we add many random reals to $M$ by forcing, then any uncountable subset of random reals will be non-measurable in the extension. See \cite{Jec}, pp. 535-536, for details.
\end{proof}

We can now piece these results together to prove our main theorem:

\begin{theorem}\label{consistencyproof}
It is consistent with \emph{ZFC} that every infinite compact group has a non-measurable subgroup.
\end{theorem}
\begin{proof}
Let $G$ be a compact group. By Theorem~\ref{HHMtheorem} and Lemma~\ref{CantorSpace}, we may assume that $G$, considered as a topological space, is homeomorphic to $2^\w$. By Lemma~\ref{CantorSpace2}, we may assume that the measure on $G$ is Lebesgue measure (provided we do not use the translation invariance property that is specific to Haar measure). By Lemma~\ref{randomreals}, it is consistently true that there is a nonmeasurable subset $X$ of $2^\w$ such that $\card{X} < 2^{\aleph_0}$. We will use this hypothesis to obtain a non-measurable subgroup of $G$

Let $H = \<X\>$ be the subgroup of $G$ generated by $X$. Clearly $\card{H} = \card{X} \cdot \aleph_0 = \card{X} < 2^{\aleph_0}$. $H$ cannot have measure $0$, because then every subset of $H$, including $X$, would have measure $0$. $H$ also cannot have positive measure, since then $H$ would be closed (and infinite) in $2^\w$ and thus of cardinality $2^{\aleph_0}$ (see \cite{HHM}, Proposition 1.1(b) for why $H$ would be closed). Thus $H$ is nonmeasurable.
\end{proof}

\section{The construction}\label{sec:constructions}



We now exhibit a technique for obtaining subgroups of a given compact group. In general, we do not know whether this construction produces non-measurable subgroups. We will prove that it does so in the Abelian case and leave the general case open. However, we conjecture that this technique always produces non-measurable subgroups of profinite groups; i.e., it provides a possible candidate for a solution to the question at hand.

\begin{theorem}\label{Zorn}
Let $G$ be an infinite, non-discrete group, let $p \in G \setminus \{1\}$, and let $Q$ be a proper dense subgroup of $G$.
\begin{enumerate}
\item There is a subgroup $M$ of $G$ such that $p \notin M \supseteq Q$ and, for any $x \notin M$, $p \in \<M,x\>$.
\item If $G$ is Abelian, then any such $M$ is non-measurable.
\end{enumerate}
\end{theorem}
\begin{proof}
$(1)$ Let $\p$ be the set of all subgroups $X$ of $G$ such that $p \notin X \supseteq Q$. Let $\C$ be a subset of $\p$ totally ordered by $\sub$. Then $p \notin \bigcup \C \supseteq Q$. Furthermore, $\bigcup C$ is a group because if $x,y \in \bigcup \C$ then there is some $X_x,X_y \in \C$ with $x \in X_x$ and $y \in X_y$, and since $X_x \sub X_y$ without loss of generality, we have $x^{-1}y \in X_y$, hence $x^{-1}y \in \bigcup \C$. Thus the conditions of Zorn's Lemma are satisfied: an application of Zorn's Lemma yields an element $M$ of $\p$ that is not properly contained in any other element of $\p$. This $M$ is the desired group.

$(2)$ Suppose $G$ is Abelian, and $M$ is as given above. For every $x \in G \setminus M$, $p \in \<M,x\>$. Using the assumption that $G$ is Abelian, for every $x \in G \setminus M$ there is some $n \in \Z$ such that $p \in nx+M$. Since $p \in G \setminus M$, the coset $p+M$ is not the identity of $G / M$. Since $G/M$ is Abelian, there is a character $f: G/M \to \R/\Z$ with $f(p+M)$ different from the identity. Then $\mathrm{ker}(f)$ is a subgroup of $G/M$ that does not contain $p+M$. Since $M$ is maximal among all subgroups of $G$ that contain $Q$ and not $p$, we must have $M = \mathrm{ker}(f)$. Hence $f$ is injective.

Because $p \in \<M,x\>$ for every $x \in G \setminus M$, we have $f(p+M) \in \<f(M),f(x+M)\> = \<f(x+M)\>$ for every $x \in G \setminus M$. This implies that for each $x \in G \setminus M$ there is some $n \in Z$ such that $f(x+M)^n = f(M)$. But any element of $\R/\Z$ has only finitely many roots of each order, hence only countably many roots. Thus $f(G/M)$ is countable, and $[G:M]$ is countable. Recall that $Q \sub M \neq X$, $M$ is not clopen. As discussed in the introduction, any non-clopen subgroup of countable index is non-measurable.
\end{proof}

Note that every infinite, metric, profinite group has a proper dense subgroup. This follows from Lemma~\ref{CantorSpace}: any such group has a countable dense subset $Q$, and then $\<Q\>$ is the desired subgroup.

\begin{question}
If $G$ is an infinite, metric, profinite group and $Q$ is a proper dense subgroup of $G$, does Theorem~\ref{Zorn}$(1)$ provide a non-measurable subgroup of $G$?
\end{question}

\end{document}